\documentclass[12pt]{article}
\usepackage{amsfonts}
\usepackage{amsmath}

\newtheorem{theorem}{Theorem}

\newtheorem{lemma}[theorem]{Lemma}

\newtheorem{proposition}[theorem]{Proposition}

\newenvironment{proof}[1][Proof]{\textbf{#1.} }{\ \rule{0.5em}{0.5em}}

\begin{document}

\title{Central limit theorem for the     modulus of
 continuity of the Brownian local time  in $L^3(\mathbb{R})$}
\author{Yaozhong Hu, David Nualart\thanks{
D. Nualart is supported by the NSF grant DMS0604207.} \\
Department of Mathematics \\
University of Kansas \\
Lawrence, Kansas, 66045 USA}
\date{}
\maketitle
\begin{abstract}
The purpose of this note is to prove a central limit theorem for the
$L^3$-modulus of continuity of the Brownian local time using techniques of stochastic analysis. The main
ingredients of the proof are an asymptotic version of Knight's
theorem and the Clark-Ocone formula for the $L^3$-modulus  of the
Brownian local time.

\end{abstract}

\section{Introduction}

Let $\{B_{t},t\geq 0\}$ be a standard one-dimensional Brownian motion. We
denote by $\{L_{t}^{x},t\geq 0,x\in \mathbb{R}\}$ a continuous version of
its local time. The following central limit theorem for the $L^{2}$ modulus
of continuity of the local time has been recently proved:
\begin{equation}
h^{-\frac{3}{2}}\left( \int_{\mathbb{R}}(L_{t}^{x+h}-L_{t} ^{x})^{2}dx-4th%
\right) \overset{\mathcal{L}}{\longrightarrow }\frac {8} {\sqrt{3}} \left( \int_{\mathbb{R}}(L_{t}^{x})^{2}dx   \right)^{\frac 12}%
\eta ,  \label{b3}
\end{equation}%
where $\eta $ is a $N(0,1)$ random variable independent of $B$  and
$\mathcal{L}$ denotes the convergence in law. This result has been
first proved in \cite{CLMR} by using the method of moments. In
\cite{HN} we gave a simple proof based on Clark-Ocone formula and an
asymptotic version of Knight's theorem (see Revuz and Yor \cite{RY},
Theorem (2.3), page 524). Another simple proof of this result with
the  techniques of stochastic analysis has been given in \cite{R1}.

The purpose of this paper is to show a central limit theorem for the
modulus of continuity in $L^{3}(\mathbb{R})$ of the local time. The
main result is the following.
\begin{theorem}   \label{thm1}
\label{th1} For each fixed $t>0$
\[
\frac{1}{h^{2}}\int_{\mathbb{R}}(L_{t}^{x+h}-L_{t}^{x})^{3}dx \overset{\mathcal{L}}{\longrightarrow }
8\sqrt{3}\left( \int_{\mathbb{R}}(L_{t}^{x})^{3}dx\right) ^{\frac 12} \eta
\]%
  as $h$ tends to zero, where $\eta $ is a normal random variable with mean zero and variance one
that is independent of $B$.
\end{theorem}

A similar central limit theorem  theorem  has been recently proved by Rosen in \cite{R2}    using the
method of moments.
 Here we  prove Theorem \ref{thm1}   using  the Clark-Ocone
stochastic integral representation   formula which allows us  to express  the random variable
\begin{equation}  \label{x1}
F_{t}^{h}=\int_{\mathbb{R}}(L_{t}^{x+h}-L_{t}^{x})^{3}dx
\end{equation}
as a stochastic integral. In comparison with the $L^2$ modulus of
continuity,  in this case  there are four different terms (instead of two),
and two of them are martingales.  Surprisingly, some of the terms of this representation converge
in $L^2(\Omega)$ to the  the derivative of the self-intersection local time
and the limits cancel out. Finally, there is a remaining martingale term
to which we can apply  the  asymptotic version of Knight's theorem.
As in the proof of    (\ref{b3}), to show the convergence of the
quadratic variation of this martingale and other asymptotic results
we make use of    Tanaka's formula for the time-reversed Brownian motion  and backward It\^{o} stochastic
integrals.   

We believe that a similar result could be established for the
modulus of continuity in $L^p(\mathbb{R})$ for an integer $p\ge 4$ using Clark-Ocone
representation formula, but the proof would be much more involved.

The paper is organized as follows. In the next section we recall
some preliminaries on Malliavin calculus. In Section 3    we
establish a stochastic integral representation for the derivative of
the self-intersection local time, which has its own interest,  and for the random variable
$F_{t}^h$ defined in (\ref{x1}).   Section 4 is devoted to the proof of Theorem
\ref{thm1}, and the Appendix contains two technical lemmas.

\section{Preliminaries on Malliavin Calculus}

Let us  recall  some basic facts on the Malliavin calculus with
respect the the Brownian motion $B=\{B_{t},t\geq 0\}$. We refer \ to
\cite{Nu06} for a complete presentation of these notions. \ We
assume that $B$ is defined on a complete probability space $(\Omega
,\mathcal{F},P)$ such that $\mathcal{F}$ is generated by $B$.
Consider the set $\mathcal{S}$ of smooth random
variables  of the form%
\begin{equation}
F=f\left( B_{t_{1}},\ldots ,B_{t_{n}}\right) ,  \label{h1}
\end{equation}%
where $\ t_{1},\ldots ,t_{n}\geq 0$, $f\in \mathcal{C}_{b}^{\infty }\left(
\mathbb{R}^{n}\right) $ (the space of bounded functions which have bounded
derivatives of all orders) and $n\in \mathbb{N}$. The derivative operator $D$
on a smooth random variable of the form (\ref{h1}) is defined by
\begin{equation*}
D_{t}F=\sum_{i=1}^{n}\frac{\partial f}{\partial x_{i}}\left(
B_{t_{1}},\ldots ,B_{t_{n}}\right) \mathbf{1}_{[0,t_{i}]}(t),
\end{equation*}%
which is an element of $L^{2}\left( \Omega \times \lbrack 0,\infty )\right) $%
. We denote by \ $\mathbb{D}^{1,2}$ the completion of $\mathcal{S}$ with
respect to the norm $\left\| F\right\| _{1,2}\ $\ given by
\begin{equation*}
\left\| F\right\| _{1,2}^{2}=E\left[ F^{2}\right] +E
\left( \int_{0}^{\infty }\left( D_{t}F\right) ^{2}dt\right) .
\end{equation*}%
The classical It\^{o} representation theorem asserts that any square
integrable random variable can be expressed as%
\begin{equation*}
F=E [F ] + \int_{0}^{\infty }u_{t}dB_{t},
\end{equation*}
where $u=\{u_{t},t\geq 0\}$ is a unique adapted process such that $E%
\left( \int_{0}^{\infty }u_{t}^{2}dt\right) <\infty $. If $F$ belongs to $%
\mathbb{D}^{1,2}$, then $u_{t}= E[D _{t}F|\mathcal{F}_{t}]$, \ where $%
\{\mathcal{F}_{t},t\geq 0\}$ is the filtration generated by $B$, and we
obtain the Clark-Ocone formula (see \cite{Oc})%
\begin{equation}
F= E[ F\mathbb{]+}\int_{0}^{\infty } E[D _{t}F|\mathcal{F}%
_{t}]dB_{t}.  \label{x2}
\end{equation}

\section{Stochastic integral representations}

Consider the   random variable $\gamma_t $
  defined rigorously as the limit in
$L^{2}(\Omega )$
\begin{equation} \label{k1}
\gamma_t=\lim_{\varepsilon \rightarrow 0} \int_{0}^{t}\int_{0}^{u}p_{\varepsilon
}^{\prime }(B_{u}-B_{s})dsdu,
\end{equation}
where $p_{\varepsilon }(x)=(2\pi \varepsilon )^{-\frac{1}{2}}\exp (-x^{2}/2\varepsilon)$.
The process  $\gamma_t$ coincides with the derivative $-\frac {d}{dy} \alpha_t(y) |_{y=0}$ of the self-intersection local time
\[
\alpha_t(y) = \int_0^t \int_0^u \delta_y(B_u-B_s) dsdu.
\]
The derivative of the self-intersection local time has been studied by Rogers and Walsh in  \cite{RW} and by Rosen
in \cite{R1}.

  We  are going to
use Clark-Ocone formula to show that the limit (\ref{k1}) exists and to provide an
integral represenation for this random variable.

\begin{lemma}   \label{lem1}
Set   $\gamma^\epsilon_t =\int_{0}^{t}\int_{0}^{u}p_{\varepsilon
}^{\prime }(B_{u}-B_{s})dsdu$. Then, $\gamma^\epsilon_t$
converges in $L^{2}(\Omega )$ as $\varepsilon $ tends to zero to the random
variable%
\[
\gamma_{t}=2\int_{0}^{t}\left(
\int_{0}^{r}p_{t-r}(B_{r}-B_{s})ds-L_{r}^{B_{r}}\right) dB_{r}.
\]
\end{lemma}

\begin{proof}
By Clark-Ocone formula applied to  $\gamma^\epsilon_t$  we
obtain the integral representation
\[
\gamma^\epsilon_t=\int_{0}^{1}E(D_{r}\gamma^\epsilon_t|\mathcal{F}_{r})dB_{r},
\]%
where $\{\mathcal{F}_{t},t\geq 0\}$ denotes the filtration generated by the
Brownian motion. Then,
\[
D_{r}\gamma^\epsilon_t=\int_{0}^{t}\int_{0}^{u}p_{\epsilon
}^{\prime \prime }(B_{u}-B_{s})\mathbf{1}_{[s,u]}(r)dsdu,
\]%
and for any $r\le t$
\begin{eqnarray*}
&&E(D_{r}\gamma^\epsilon_t|\mathcal{F}_{r})=\int_{r}^{t}%
\int_{0}^{r}p_{\epsilon +u-r}^{\prime \prime }(B_{r}-B_{s})dsdu \\
&&\quad =2\int_{r}^{t}\int_{0}^{r}\frac{\partial p_{\epsilon +u-r}}{\partial
u}(B_{r}-B_{s})dsdu \\
&&\quad =2\int_{0}^{r}(p_{\epsilon +t-r}(B_{r}-B_{s})-p_{\epsilon
}(B_{r}-B_{s}))ds.
\end{eqnarray*}%
As $\varepsilon $ tends to zero this expression converges in $L^{2}(\Omega
\times \lbrack 0,t])$ to%
\[
2\left( \int_{0}^{r}p_{t-r}(B_{r}-B_{s})ds-L_{r}^{B_{r}}\right) ,
\]%
which completes the proof.
\end{proof}

 Let us now obtain a stochastic integral representation
for the third integrated moment $F_{t}^{h}=\int_{\mathbb{R}%
}(L_{t}^{x+h}-L_{t}^{x})^{3}dx$. Notice first that $E(F_{t}^{h})=0$ because $%
F_{t}^{h}$ is an odd functional of the Brownian motion.

\begin{proposition}
We have $F_{t}^{h}=\int_{0}^{t}\Phi _{r}dB_{r}$, where $\Phi
_{r}=\sum_{i=1}^{4}\Phi _{r}^{(i)}$, and
\begin{eqnarray*}
\Phi _{r}^{(1)} &=&6\ \int_{\mathbb{R}}\left( L_{r}^{z+h}-L_{r}^{z}\right)
^{2}\mathbf{1}_{[0,h]}(B_{r}-z)dz \\
\Phi _{r}^{(2)} &=&-6\int_{\mathbb{R}}\int_{0}^{h}\ \left(
L_{r}^{z+h}-L_{r}^{z}\right) ^{2}p_{t-r}(B_{r}-z-y)dydz \\
\Phi _{r}^{(3)} &=&\frac{12h}{\sqrt{2\pi }}\int_{0}^{r}\int_{-h}^{h}%
\int_{\frac{h^2} { t-r}}^{\infty }p_{t-r-\frac{h^{2}}{z}}(B_{r}-B_{s}+y)z^{-\frac{3}{2}%
}(1-e^{-\frac{z}{2}})dzdyds \\
\Phi _{r}^{(4)}
&=&-\frac{12h}{\sqrt{2\pi}}\int_{0}^{r}\mathbf{1}_{[-h,h]}(B_{r}-B_{s})ds \int_{\frac{h^2} { t-r}}^{\infty } z^{-\frac{3}{2}%
}(1-e^{-\frac{z}{2}})dz  .
\end{eqnarray*}
\end{proposition}

\begin{proof}
Let us write%
\begin{eqnarray*}
F_{t}^{h} &=&\lim_{\varepsilon \rightarrow 0}\int_{\mathbb{R}}\left(
\int_{0}^{t}\left[ p_{\varepsilon }(B_{s}-x-h)-p_{\varepsilon }(B_{s}-x)%
\right] ds\right) ^{3}dx \\
&=&\lim_{\varepsilon \rightarrow 0}\int_{\mathbb{R}}\int_{[0,t]^{3}}%
\prod_{i=1}^{3}\left( \left[ p_{\varepsilon }(B_{s_{i}}-x-h)-p_{\varepsilon
}(B_{s_{i}}-x)\right] \right) dsdx \\
&=&6\lim_{\varepsilon \rightarrow 0}\int_{D}\int_{\mathbb{R}}\prod_{i=1}^{3}%
\left[ p_{\varepsilon }(B_{s_{i}}-x-h)-p_{\varepsilon }(B_{s_{i}}-x)\right]
dxds,
\end{eqnarray*}%
where $D=\{(s_{1},s_{2},s_{3})\in \lbrack 0,t]^{3}:s_{1}<s_{2}<s_{3}\}$. We
can pass to the limit  as $\epsilon$ tends to zero the first factor $p_{\varepsilon
}(B_{s_{1}}-x-h)-p_{\varepsilon }(B_{s_{1}}-x)$, and we obtain%
\begin{eqnarray*}
F_{t}^{h} &=&\lim_{\varepsilon \rightarrow 0}6\int_{D}\Big\{\left[
p_{\varepsilon }(B_{s_{2}}-B_{s_{1}})-p_{\varepsilon }(B_{s_{2}}-B_{s_{1}}+h)%
\right]  \\
&&\times \left[ p_{\varepsilon }(B_{s_{3}}-B_{s_{1}})-p_{\varepsilon
}(B_{s_{3}}-B_{s_{1}}+h)\right]  \\
&&-\left[ p_{\varepsilon }(B_{s_{2}}-B_{s_{1}}-h)-p_{\varepsilon
}(B_{s_{2}}-B_{s_{1}})\right]  \\
&&\times \left[ p_{\varepsilon }(B_{s_{3}}-B_{s_{1}}-h)-p_{\varepsilon
}(B_{s_{3}}-B_{s_{1}})\right] \Big\}ds \\
&=&\lim_{\varepsilon \rightarrow 0}6\int_{D}\Phi _{\varepsilon }(s)ds.
\end{eqnarray*}

We are going to apply the Clark-Ocone formula to the random variable $%
\int_{D}\Phi _{\varepsilon }(s)ds$. Fix $r\in \lbrack 0,t]$. We need to
compute $\int_{D}E$ $\left( D_{r}\left[ \Phi _{\varepsilon }(s)\right] |%
\mathcal{F}_{r}\right) ds$.  To do this we decompose the \ region
$D$, up to a set  of zero Lebesgue measure,   as   $D=D_{0}\cup
D_{1}\cup D_{2}$, where
\begin{eqnarray*}
D_{1} &=&\{s:0\leq s_{1}<s_{2}<r<s_{3}\leq t\}, \\
D_{2} &=&\{s:0\leq s_{1}<r<s_{2}<s_{3}\leq t\},
\end{eqnarray*}%
and $D_{0}= \{s: 0\le r\le s_1\} \cup \{s: s_3 \le r\le t\}$. Notice that on $D_{0}$%
, $D_{r}\left[ \Phi _{\varepsilon }(s)\right] =0$.

\noindent \textbf{Step 1. }For the region $D_{1}$ we obtain%
\begin{eqnarray*}
E\left( D_{r}\left[ \Phi _{\varepsilon }(s)\right] |\mathcal{F}_{r}\right)
&=&E(\{\left[ p_{\varepsilon }(B_{s_{2}}-B_{s_{1}})-p_{\varepsilon
}(B_{s_{2}}-B_{s_{1}}+h)\right]  \\
&&\times \left[ p_{\varepsilon }^{\prime
}(B_{s_{3}}-B_{s_{1}})-p_{\varepsilon }^{\prime }(B_{s_{3}}-B_{s_{1}}+h)%
\right]  \\
&&-\ \ \left[ p_{\varepsilon }(B_{s_{2}}-B_{s_{1}}-h)-p_{\varepsilon
}(B_{s_{2}}-B_{s_{1}})\right]  \\
&&\times \left[ p_{\varepsilon }^{\prime
}(B_{s_{3}}-B_{s_{1}}-h)-p_{\varepsilon }^{\prime }(B_{s_{3}}-B_{s_{1}})%
\right] \}|\mathcal{F}_{r}) \\
&=&\left[ p_{\varepsilon }(B_{s_{2}}-B_{s_{1}})-p_{\varepsilon
}(B_{s_{2}}-B_{s_{1}}+h)\right]  \\
&&\times \left[ p_{\varepsilon +s_{3}-r}^{\prime
}(B_{r}-B_{s_{1}})-p_{\varepsilon +s_{3}-r}^{\prime }(B_{r}-B_{s_{1}}+h)%
\right]  \\
&&-\ \ \left[ p_{\varepsilon }(B_{s_{2}}-B_{s_{1}}-h)-p_{\varepsilon
}(B_{s_{2}}-B_{s_{1}})\right]  \\
&&\times \left[ p_{\varepsilon +s_{3}-r}^{\prime
}(B_{r}-B_{s_{1}}-h)-p_{\varepsilon +s_{3}-r}^{\prime }(B_{r}-B_{s_{1}})%
\right] .
\end{eqnarray*}%
We can write this in the following form%
\begin{eqnarray*}
&&E\left( D_{r}\left[ \Phi _{\varepsilon }(s)\right] |\mathcal{F}_{r}\right)
\\
&=&\int_{0}^{h}\left[ p_{\varepsilon }(B_{s_{2}}-B_{s_{1}}+h)-p_{\varepsilon
}(B_{s_{2}}-B_{s_{1}})\right] p_{\varepsilon +s_{3}-r}^{\prime \prime
}(B_{r}-B_{s_{1}}+y)dy \\
&&-\int_{0}^{h}\left[ p_{\varepsilon }(B_{s_{2}}-B_{s_{1}})-p_{\varepsilon
}(B_{s_{2}}-B_{s_{1}}-h)\right] p_{\varepsilon +s_{3}-r}^{\prime \prime
}(B_{r}-B_{s_{1}}-y)dy \\
&=&2\int_{0}^{h}\left[ p_{\varepsilon
}(B_{s_{2}}-B_{s_{1}}+h)-p_{\varepsilon }(B_{s_{2}}-B_{s_{1}})\right] \frac{%
\partial p_{\varepsilon +s_{3}-r}}{\partial s_{3}}(B_{r}-B_{s_{1}}+y)dy \\
&&-2\int_{0}^{h}\left[ p_{\varepsilon }(B_{s_{2}}-B_{s_{1}})-p_{\varepsilon
}(B_{s_{2}}-B_{s_{1}}-h)\right] \frac{\partial p_{\varepsilon +s_{3}-r}}{%
\partial s_{3}}(B_{r}-B_{s_{1}}-y)dy.
\end{eqnarray*}%
Integrating with respect to the variable $s_3  $ yields%
\begin{eqnarray*}
&&\int_{D}E\left( D_{r}\left[ \Phi _{\varepsilon }(s)\right] |\mathcal{F}%
_{r}\right) ds   \\
&=& 2\int_{0\leq s_{1}<s_{2}\leq r}\int_{0}^{h}\left[ p_{\varepsilon
}(B_{s_{2}}-B_{s_{1}}+h)-p_{\varepsilon }(B_{s_{2}}-B_{s_{1}})\right]  \\
&&\times \left[ p_{\varepsilon +t-r}(B_{r}-B_{s_{1}}+y)-p_{\varepsilon
}(B_{r}-B_{s_{1}}+y)\right] dyds_{1}ds_{2} \\
&&-2\int_{0\leq s_{1}<s_{2}\leq r}\int_{0}^{h}\left[ p_{\varepsilon
}(B_{s_{2}}-B_{s_{1}})-p_{\varepsilon }(B_{s_{2}}-B_{s_{1}}-h)\right]  \\
&&\times \left[ p_{\varepsilon +t-r}(B_{r}-B_{s_{1}}-y)-p_{\varepsilon
}(B_{r}-B_{s_{1}}-y)\right] dyds_{1}ds_{2}.
\end{eqnarray*}%
This expression can be written in terms of the local time:%
\begin{eqnarray*}
&&\int_{D}E\left( D_{r}\left[ \Phi _{\varepsilon }(s)\right] |\mathcal{F}%
_{r}\right) ds  \\
&=&2\int_{\mathbb{R}^{2}}\int_{0}^{h}\int_{0}^{r}\left[ p_{\varepsilon
}(x-z+h)-p_{\varepsilon }(x-z)\right] (L_{r}^{x}-L_{s}^{x})L_{ds}^{z} \\
&&\times \left[ p_{\varepsilon +t-r}(B_{r}-z+y)-p_{\varepsilon }(B_{r}-z+y)%
\right] dydxdz \\
&&-2\int_{\mathbb{R}^{2}}\int_{0}^{h}\int_{0}^{r}\left[ p_{\varepsilon
}(x-z)-p_{\varepsilon }(x-z-h)\right] (L_{r}^{x}-L_{s}^{x})L_{ds}^{z} \\
&&\times \left[ p_{\varepsilon +t-r}(B_{r}-z-y)-p_{\varepsilon }(B_{r}-z-y)%
\right] dydxdz.
\end{eqnarray*}%
We make the change of variables $y\rightarrow h-y$ \ and $z\rightarrow z+h$
in the last integral and we obtain%
\begin{eqnarray*}
&&\int_{D}E\left( D_{r}\left[ \Phi _{\varepsilon }(s)\right] |\mathcal{F}%
_{r}\right) ds \\
&=&2\int_{\mathbb{R}^{2}}\int_{0}^{h}\int_{0}^{r}\left[ p_{\varepsilon
}(x-z+h)-p_{\varepsilon }(x-z)\right] (L_{r}^{x}-L_{s}^{x})\left(
L_{ds}^{z}-L_{ds}^{z-h}\right)  \\
&&\times \left[ p_{\varepsilon +t-r}(B_{r}-z+y)-p_{\varepsilon }(B_{r}-z+y)%
\right] dydxdz.
\end{eqnarray*}%
Taking the limit as $\varepsilon $ tends to zero in $L^{2}(\Omega \times
\lbrack 0,t])$ and integrating in $x$  yields%
\begin{eqnarray*}
&&\lim_{\epsilon \rightarrow 0}\int_{D}E\left( D_{r}\left[ \Phi
_{\varepsilon }(s)\right] |\mathcal{F}_{r}\right) ds \\
&=&2\int_{\mathbb{R}}\int_{0}^{h}\int_{0}^{r}\left(
L_{r}^{z-h}-L_{s}^{z-h}-L_{r}^{z}+L_{s}^{z}\right) \left(
L_{ds}^{z}-L_{ds}^{z-h}\right)  \\
&&\times \left[ p_{t-r}(B_{r}-z+y)-\delta _{0}(B_{r}-z+y)\right] dydz \\
&=&\int_{\mathbb{R}}\ \ \left( L_{r}^{z}-L_{r}^{z-h}\right) ^{2}\ \ \mathbf{1%
}_{[0,h]}(z-B_{r})dz-\int_{\mathbb{R}}\int_{0}^{h}\ \left(
L_{r}^{z}-L_{r}^{z-h}\right) ^{2}p_{t-r}(B_{r}-z+y)dydz \\
&=&\int_{\mathbb{R}}\ \ \left( L_{r}^{z+h}-L_{r}^{z}\right) ^{2}\ \ \mathbf{1%
}_{[0,h]}(B_{r}-z)dz-\int_{\mathbb{R}}\int_{0}^{h}\ \left(
L_{r}^{z+h}-L_{r}^{z}\right) ^{2}p_{t-r}(B_{r}-z-y)dydz \\
&:=&\Psi _{r}^{(1)}.
\end{eqnarray*}

\noindent \textbf{Step 2. }For the region $D_{2}$ we obtain%
\begin{eqnarray*}
&&E\left( D_{r}\left[ \Phi _{\varepsilon }(s)\right] |\mathcal{F}_{r}\right)
=E(\{\left[ p_{\varepsilon }^{\prime }(B_{s_{2}}-B_{s_{1}})-p_{\varepsilon
}^{\prime }(B_{s_{2}}-B_{s_{1}}+h)\right]  \\
&&\times \left[ p_{\varepsilon }(B_{s_{3}}-B_{s_{1}})-p_{\varepsilon
}(B_{s_{3}}-B_{s_{1}}+h)\right]  \\
&&+\left[ p_{\varepsilon }(B_{s_{2}}-B_{s_{1}})-p_{\varepsilon
}(B_{s_{2}}-B_{s_{1}}+h)\right] \left[ p_{\varepsilon }^{\prime
}(B_{s_{3}}-B_{s_{1}})-p_{\varepsilon }^{\prime }(B_{s_{3}}-B_{s_{1}}+h)%
\right]  \\
&&-\ \ \left[ p_{\varepsilon }^{\prime
}(B_{s_{2}}-B_{s_{1}}-h)-p_{\varepsilon }^{\prime }(B_{s_{2}}-B_{s_{1}})%
\right] \left[ p_{\varepsilon }(B_{s_{3}}-B_{s_{1}}-h)-p_{\varepsilon
}(B_{s_{3}}-B_{s_{1}})\right]  \\
&&-\ \left[ p_{\varepsilon }(B_{s_{2}}-B_{s_{1}}-h)-p_{\varepsilon
}(B_{s_{2}}-B_{s_{1}})\right] \left[ p_{\varepsilon }^{\prime
}(B_{s_{3}}-B_{s_{1}}-h)-p_{\varepsilon }^{\prime }(B_{s_{3}}-B_{s_{1}})%
\right] \}|\mathcal{F}_{r}).
\end{eqnarray*}
We take first the conditional expectation with respect to
$\mathcal{F}_{s_2}$ which contains $\mathcal{F}_r$ and we  obtain
\begin{eqnarray*}
&&E\left( D_{r}\left[ \Phi _{\varepsilon }(s)\right] |\mathcal{F}_{r}\right)
 = E(\{\left[ p_{\varepsilon }^{\prime }(B_{s_{2}}-B_{s_{1}})-p_{\varepsilon
}^{\prime }(B_{s_{2}}-B_{s_{1}}+h)\right]    \\
&& \quad  \times \left[ p_{\varepsilon
+s_{3}-s_{2}}(B_{s_{2}}-B_{s_{1}})-p_{\varepsilon
+s_{3}-s_{2}}(B_{s_{2}}-B_{s_{1}}+h)\right]  \\
&&+\left[ p_{\varepsilon }(B_{s_{2}}-B_{s_{1}})-p_{\varepsilon
}(B_{s_{2}}-B_{s_{1}}+h)\right]   \\
&&\quad\times   \left[ p_{\varepsilon +s_{3}-s_{2}}^{\prime
}(B_{s_{2}}-B_{s_{1}})-p_{\varepsilon +s_{3}-s_{2}}^{\prime
}(B_{s_{2}}-B_{s_{1}}+h)\right]  \\
&&-  \left[ p_{\varepsilon }^{\prime }(B_{s_{2}}-B_{s_{1}}-h)-p_{\varepsilon
}^{\prime }(B_{s_{2}}-B_{s_{1}})\right]   \\
&& \quad \times  \left[ p_{\varepsilon
+s_{3}-s_{2}}(B_{s_{2}}-B_{s_{1}}-h)-p_{\varepsilon
+s_{3}-s_{2}}(B_{s_{2}}-B_{s_{1}})\right]  \\
&&-\left[ p_{\varepsilon }(B_{s_{2}}-B_{s_{1}}-h)-p_{\varepsilon
}(B_{s_{2}}-B_{s_{1}})\right]  \\
&& \quad \times \left[ p_{\varepsilon +s_{3}-s_{2}}^{\prime
}(B_{s_{2}}-B_{s_{1}}-h)-p_{\varepsilon +s_{3}-s_{2}}^{\prime
}(B_{s_{2}}-B_{s_{1}})\right] \}|\mathcal{F}_{r}).
\end{eqnarray*}%
As a consequence,
\begin{eqnarray*}
&&E\left( D_{r}\left[ \Phi _{\varepsilon }(s)\right] |\mathcal{F}_{r}\right)
\\
&=&\ \int_{\mathbb{R}}dyp_{s_{2}-r}(y)\{\left[ p_{\varepsilon }^{\prime
}(B_{r}-B_{s_{1}}+y)-p_{\varepsilon }^{\prime }(B_{r}-B_{s_{1}}+y+h)\right]
\\
&&\times \left[ p_{\varepsilon
+s_{3}-s_{2}}(B_{r}-B_{s_{1}}+y)-p_{\varepsilon
+s_{3}-s_{2}}(B_{r}-B_{s_{1}}+h+y)\right]  \\
&&+\left[ p_{\varepsilon }(B_{r}-B_{s_{1}}+y)-p_{\varepsilon
}(B_{r}-B_{s_{1}}+y+h)\right]  \\
&&\times \left[ p_{\varepsilon +s_{3}-s_{2}}^{\prime
}(B_{r}-B_{s_{1}}+y)-p_{\varepsilon +s_{3}-s_{2}}^{\prime
}(B_{r}-B_{s_{1}}+h+y)\right]  \\
&&-\left[ p_{\varepsilon }^{\prime }(B_{r}-B_{s_{1}}+y-h)-p_{\varepsilon
}^{\prime }(B_{r}-B_{s_{1}}+y)\right]  \\
&&\times \left[ p_{\varepsilon
+s_{3}-s_{2}}(B_{r}-B_{s_{1}}+y-h)-p_{\varepsilon
+s_{3}-s_{2}}(B_{r}-B_{s_{1}}+y)\right]  \\
&&-\left[ p_{\varepsilon }(B_{r}-B_{s_{1}}+y-h)-p_{\varepsilon
}(B_{r}-B_{s_{1}}+y)\right]  \\
&&\times \left[ p_{\varepsilon +s_{3}-s_{2}}^{\prime
}(B_{r}-B_{s_{1}}+y-h)-p_{\varepsilon +s_{3}-s_{2}}^{\prime
}(B_{r}-B_{s_{1}}+y)\right] \} .
\end {eqnarray*}
Integrating by parts we obtain
\begin{eqnarray*}
&&E\left( D_{r}\left[ \Phi _{\varepsilon }(s)\right] |\mathcal{F}_{r}\right) \\
& &=-\int_{\mathbb{R}}dyp_{s_{2}-r}^{\prime }(y)\{\left[ p_{\varepsilon
}(B_{r}-B_{s_{1}}+y)-p_{\varepsilon }(B_{r}-B_{s_{1}}+y+h)\right]  \\
&&\times \left[ p_{\varepsilon
+s_{3}-s_{2}}(B_{r}-B_{s_{1}}+y)-p_{\varepsilon
+s_{3}-s_{2}}(B_{r}-B_{s_{1}}+h+y)\right]  \\
&&-\left[ p_{\varepsilon }(B_{r}-B_{s_{1}}+y-h)-p_{\varepsilon
}(B_{r}-B_{s_{1}}+y)\right]  \\
&&\times \left[ p_{\varepsilon
+s_{3}-s_{2}}(B_{r}-B_{s_{1}}+y-h)-p_{\varepsilon
+s_{3}-s_{2}}(B_{r}-B_{s_{1}}+y)\right] \}.
\end{eqnarray*}%
Letting $\varepsilon $ tend to zero we obtain%
\begin{eqnarray*}
&&\lim_{\epsilon \rightarrow 0}E\left( D_{r}\left[ \Phi _{\varepsilon }(s)%
\right] |\mathcal{F}_{r}\right)  \\
&=&\left[ p_{s_{3}-s_{2}}(0)-p_{s_{3}-s_{2}}(h)\right] \left[
p_{s_{2}-r}^{\prime }(B_{r}-B_{s_{1}}+h)-p_{s_{2}-r}^{\prime
}(B_{r}-B_{s_{1}}-h)\right]  \\
&=&\left[ p_{s_{3}-s_{2}}(0)-p_{s_{3}-s_{2}}(h)\right]
\int_{-h}^{h}p_{s_{2}-r}^{\prime \prime }(B_{r}-B_{s_{1}}+y)dy.
\end{eqnarray*}%
Hence,%
\begin{eqnarray}
&&\lim_{\epsilon \rightarrow 0}\int_{D_{2}}E\left( D_{r}\left[ \Phi
_{\varepsilon }(s)\right] |\mathcal{F}_{r}\right) ds  \nonumber \\
&&\quad =2\int_{0}^{r}ds_{1}\int_{r}^{t}ds_{2}\left( \int_{s_{2}}^{t}\left[
p_{s_{3}-s_{2}}(0)-p_{s_{3}-s_{2}}(h)\right] ds_{3}\right)   \nonumber \\
&&\times \int_{-h}^{h}\frac{\partial p_{s_{2}-r}}{\partial s_{2}}%
(B_{r}-B_{s_{1}}+y)dy:=\Psi _{r}^{(2)}.  \label{e2}
\end{eqnarray}%
We have
\begin{eqnarray}
\int_{s_{2}}^{t}\left[ p_{s_{3}-s_{2}}(0)-p_{s_{3}-s_{2}}(h)\right] ds_{3}
&=&\frac{1}{\sqrt{2\pi }}\int_{0}^{t-s_{2}}\frac{1}{\sqrt{s}}\left( 1-e^{-%
\frac{h^{2}}{2s}}\right) ds  \nonumber \\
&=&\frac{h}{\sqrt{2\pi }}\int_{\frac{h^{2}}{t-s_{2}}}^{\infty }z^{-\frac{3}{2%
}}(1-e^{-\frac{z}{2}})dz.  \label{e1}
\end{eqnarray}%
Substituting (\ref{e1}) into (\ref{e2}) yields
\[
\Psi _{r}^{(2)}=\frac{2h}{\sqrt{2\pi }}\int_{-h}^{h}\int_{ \frac{h^2} {t-r}}^{\infty
}\int_{r}^{t-\frac{h^{2}}{z}}   \int_0^r\frac{\partial p_{u-r}}{\partial u}%
(B_{r}-B_{s}+y)\times z^{-\frac{3}{2}}(1-e^{-\frac{z}{2}})dsdudzdy.
\]%
Now we integrate in the variable $u$ and we obtain
\begin{eqnarray*}
\Psi _{r}^{(2)}&=&\frac{2h}{\sqrt{2\pi }}\int_{-h}^{h}\int_{ \frac{h^2} {t-r}}^{\infty }  \int_0^rp_{t-r-%
\frac{h^{2}}{z}}(B_{r}-B_{s}+y)z^{-\frac{3}{2}}(1-e^{-\frac{z}{2}})dsdzdy\\
&&-\frac{2h}{\sqrt{2\pi}}\int_0^r
\mathbf{1}_{[-h,h]}(B_{r}-B_{s})ds  \int_{\frac{h^2} { t-r}}^{\infty } z^{-\frac{3}{2}%
}(1-e^{-\frac{z}{2}})dz.
\end{eqnarray*}%
Thus,
\[
M_{t}=6\int_{0}^{t}\Psi _{r}^{(1)}dB_{r}+6\int_{0}^{t}\Psi _{r}^{(2)}dB_{r},
\]%
which completes the proof.
\end{proof}

\section{Proof of Theorem \ref{thm1}}

 The proof will be done in several steps. Along the proof we will denote by $%
C $ a generic constant, which may be different from line to line.

\textit{Step 1 } Notice first that by Lemma 2 and  the equation
\begin{equation}  \label{h2}
\frac 1{\sqrt{2\pi}} \int_0^\infty z^{-\frac 32} (1-e^{-\frac
z2})dz=1,
\end{equation}
we obtain that 
$h^{-2}\int_{0}^{t}(\Phi _{r}^{(3)}+\Phi _{r}^{(4)})dB_{r}$
converges in $L^{2}(\Omega )$ to $12\gamma_{t}$.

\textit{Step 2 } In order to handle the term $h^{-2}\int_{0}^{t}(\Phi _{r}^{(1)}+\Phi _{r}^{(2)})dB_{r}$ we consider the
function
\[
\phi _{h}(\xi )=\int_{0}^{h}p_{t-r}(\xi -y)dy-\mathbf{1}_{[0,h]}(\xi ),
\]%
and we can write
\begin{eqnarray*}
\Phi ^{(1)}_r+\Phi ^{(2)}_r &=&-6\int_{\mathbb{R}} \left( L_{r}^{z+h}-L_{r}^{z}\right) ^{2}\phi_h (B_r-z) dz\\
&=& -6\int_{\mathbb{R}}\left(
L_{r}^{B_{r}-x+h}-L_{r}^{B_{r}-x}\right) ^{2}\phi _{h}(x)dx.
\end{eqnarray*}
Applying Tanaka's formula to the time reversed Brownian motion $%
\{B_{r}-B_{s},0\leq s\leq r\}$, we obtain
\begin{eqnarray*}
 \frac 12 \left( L_{r}^{B_{r}-x+h}-L_{r}^{B_{r}-x}  \right)
&=&-(-x+h)^{+}+(-x)^{+}+(B_{r}-x+h)^{+}-(B_{r}-x)^{+} \\
&&+\int_{0}^{r}\mathbf{1}_{\{B_{r}-x-B_{s}>0\}}dB_{s}-\int_{0}^{r}\mathbf{1}%
_{\{B_{r}-x+h-B_{s}>0\}}d\widehat{B}_{s} \\
&=&-(-x+h)^{+}+(-x)^{+}+(B_{r}-x+h)^{+}-(B_{r}-x)^{+} \\
&&-\int_{0}^{r}\mathbf{1}_{[-h,0]}(B_{r}-B_{s}-x)d\widehat{B}_{s},
\end{eqnarray*}%
where $d\widehat{B}_{s}$ denotes the backward It\^{o} integral. Clearly
\[
h^{-4}E\int_{0}^{t}\left( \int_{\mathbb{R}}
\left[ (-x+h)^{+}-(-x)^{+}-(B_{r}-x+h)^{+}+(B_{r}-x)^{+} \right] ^{2}\phi
_{h}(x)dx\right) ^{2}dr\rightarrow 0,
\]%
as $h$ tends to zero. So, the only term that gives a nonzero contribution is
\[
 24\int_{\mathbb{R}}\left( \int_{0}^{r}\mathbf{1}_{[-h,0]}(B_{r}-B_{s}-x)d%
\widehat{B}_{s}\right) ^{2}\phi _{h}(x)dx.
\]%
We are going to use the following   notation
\begin{eqnarray}
\delta ^h _{r,s}(x)&=& \mathbf{1}_{[-h,0]} (B_r-B_s -x),   \label{l1}\\
A^h_{\sigma ,r}(x) &=&\int_{\sigma }^{r}\delta _{r,s }^{h}(x)d\widehat{B}%
_{s}\,,   \label{l2}
\end{eqnarray}
where $0<\sigma<r$. With this notation we  want to find the limit in distribution of 
\[
Y_h:=-24h^{-2} \int_0^t  \left( \int_{\mathbb{R}} (A^h_{0,r}(x))^2 \phi_h(x) dx \right) dB_r,
\]
as $h$ tends to zero.    By It\^{o}'s formula,
\[
(A_{0,r}^{h}(x))^2=2\int_{0}^{r}A^h_{\sigma ,r}(x)\delta _{r,\sigma }^{h}(x)d%
\widehat{B}_{\sigma }+\int_{0}^{r}\delta _{r,s}^{h}(x)ds.
\]%
Therefore,%
\begin{eqnarray}
   Y_h &=&   -48h^{-2} \int_0^t \left(\int_{\mathbb{R}}\left( \int_{0}^{r}A^h_{\sigma ,r}(x)\delta
_{r,\sigma }^{h}(x)d\widehat{B}_{\sigma }\right) \phi _{h}(x)dx \right) dB_r  \nonumber \\
&&-24h^{-2} \int_0 ^t \left( \int_{\mathbb{R}}   \int_{0}^{r}\delta
_{r,s}^{h}(x)  \phi _{h}(x)dsdx \right) dB_r .  \label{h3}
\end{eqnarray}%
Notice that, by Lemma \ref{lem1}
\begin{eqnarray*}
&&-24h^{-2}\int_{0}^{t}\left( \int_{\mathbb{R}}\int_{0}^{r}\delta
_{r,s}^{h}(x)ds\phi _{h}(x)dsdx\right) dB_{r} \\
&&= - 24h^{-2} \int_0^t  \left(  \int_\mathbb{R} \int_0^r \mathbf{1}
_{[-h,0]} (B_r-B_s-x) \left( \int_0^h p_{t-r} (x-y) dy-
\mathbf{1}_{[0,h]} (x)   \right) dsdx  \right) dB_r
\end{eqnarray*}
converges in $L^{2}(\Omega )$ to $-12\gamma_{t}$, which cancels with the limit obtained in Step 1.

To handle the first summand in the right-hand side of (\ref{h3}) we
make the decomposition
\[
\int_{\mathbb{R}}\left( \int_{0}^{r}A^h_{\sigma ,r}(x)\delta _{r,\sigma
}^{h}(x)d\widehat{B}_{\sigma }\right) \phi _{h}(x)dx=\Gamma _{r,h}-\Delta
_{r,h},
\]%
where
\begin{equation}
\Gamma _{r,h}=\int_{0}^{h}\int_{\mathbb{R}}\left(
\int_{0}^{r}A^h_{\sigma ,r}(x)\delta _{r,\sigma
}^{h}(x)d\widehat{B}_{\sigma }\right) p_{t-r}(x-y)dxdy,  \label{h4}
\end{equation}%
and
\begin{equation}
\Delta _{r,h}=\int_{0}^{h}\left( \int_{0}^{r}A^h_{\sigma ,r}(x)\delta
_{r,\sigma }^{h}(x)d\widehat{B}_{\sigma }\right) dx,  \label{h5}
\end{equation}

\textit{Step 3 }
We claim that%
\begin{equation}
\lim_{h\rightarrow 0}h^{-4}E(\Gamma _{r,h}^{2})=0,  \label{h6}
\end{equation}%
which implies that $ h^{-2} \int_0^t \Gamma_{r,h} dB_r$ converges to $0$ in $L^2(\Omega)$ as $h$ tends  to zero.  Let us prove (\ref{h6}). We can write
\begin{eqnarray*}
E(\Gamma _{r,h}^{2}) &=&\int_{0}^{r}\left( \int_{0}^{h}\int_{\mathbb{R}%
}A^h_{\sigma ,r}(x)\delta _{r,\sigma }^{h}(x)p_{t-r}(x-y)dxdy\right)
^{2}d\sigma  \\
&\leq &E\Bigg(\sup_{0\leq \sigma \leq r\leq t}\sup_{x\in \mathbb{R}%
}|A^h_{\sigma ,r}(x)|^{2} \\
&&\times \int_{0}^{r}\left( \int_{0}^{h}\int_{\mathbb{R}}\delta _{r,\sigma
}^{h}(x)p_{t-r}(x-y)dxdy\right) ^{2}d\sigma \Bigg).
\end{eqnarray*}%
Clearly, for any $p\geq 1$,
\[
h^{-4}\int_{0}^{r}\left( \int_{0}^{h}\int_{\mathbb{R}}\delta _{r,\sigma
}^{h}(x)p_{t-r}(x-y)dxdy\right) ^{2}d\sigma
\]%
converges in $L^{p}(\Omega )$ to
\[
\int_{0}^{r}p_{t-r}(B_{r}-B_{s})^{2}ds,
\]%
and, on the other hand,  by Lemma  \ref{lema1} in the Appendix,   $\|\sup_{0\leq \sigma \leq r\leq t}\sup_{x\in \mathbb{R} 
}|A^h_{\sigma ,r}(x)|   \|_p$  converges to zero as $h$ tend to zero, for any $p\ge 2$. 
  This completes the proof of (\ref{h6}).

\textit{Step 4 }
Finally, we will discuss the limit of the martingale
\[
M_{t}^{h}=48h^{-2}\int_{0}^{t}\Delta
_{r,h}dB_{r},
\]
 where $\Delta_{r,h}$ is defined in (\ref{h5}).  From the asymptotic version of Knight's theorem (see Revuz and Yor \cite{RY}%
, Theorem  2.3  page  524) it suffices to show the following
convergences in probability.
\begin{equation}
\langle M^{h},B\rangle _{t}\rightarrow 0,  \label{f1}
\end{equation}%
in probability as $h$ tends to zero, uniformly in compact sets, and
\begin{equation}
\langle M^{h}\rangle _{t}\rightarrow 192\int_{\mathbb{R}} (L^x_t)^3 dx.
\label{f2}
\end{equation}%
In fact, let $B^h$ be the Brownian motion such that  $ M^h _t =
B^h_{\langle M^h \rangle_t}$.  Then, from   Theorem  2.3  page 524
in \cite{RY}, and the convergences  (\ref{f1}) and (\ref{f2}), we
deduce that $(B, B^h,  \langle M^h \rangle_t)$ converges in
distribution to $(B, \beta,  192\int_{\mathbb{R}} (L^x_t)^3 dx)$,
where $\beta$ is a Brownian motion independent of $B$. This implies
that  $M_t^h=B^h_{\langle M^h \rangle_t}$ converges in distribution
to $\beta_ {192\int_{\mathbb{R}} (L^x_t)^3 dx}$, which yields the
desired result.

Exchanging the order of integration we can write  $\Delta _{r,h}$
  as
\begin{equation} \label{d1}
\Delta _{r,h}=\int_{0}^{r}\left( \int_{0}^{h}A^h_{\sigma ,r}(x)\delta
_{r,\sigma }^{h}(x)dx\right) d\widehat{B}_{\sigma }=\int_{0}^{r}\Psi
_{r,\sigma }^{h}d\widehat{B}_{\sigma },
\end{equation}
where%
\begin{equation} \label{d2}
\Psi _{r,\sigma }^{h}=\int_{0}^{h}A^h_{\sigma ,r}(x)\delta _{r,\sigma
}^{h}(x)dx.
\end{equation}

\textit{Step 5 } Let us prove (\ref{f1}).  In order to establish the uniform convergence with respect to $t$ in a compact set,  we are going to obtain estimates for the increments of the process   $\langle M^{h},B\rangle _{t} $, and to make use of the  Garsia-Rodemich-Rumsey to  estimate the supremum in $t$.  For any $p\ge 2$ we have, by Burkholder's inequality
\begin{eqnarray*}
&& E\left| \langle M^{h},B\rangle _{t}-\langle M^{h},B\rangle _{s} \right|^{p} \\
&&\le c_p h^{-2p}E \left|\int_{0}^{s}\left(
\int_{s }^{t}\int_{0}^{h}A^h_{\sigma ,r}(x)\delta _{r,\sigma
}^{h}(x)dxdr\right) ^{2}d\sigma  \right|^{\frac p2}  \\
&& + c_p h^{-2p}E \left| \int_{s}^{t}\left(
\int_{\sigma }^{t}\int_{0}^{h}A^h_{\sigma ,r}(x)\delta _{r,\sigma
}^{h}(x)dxdr\right) ^{2}d\sigma   \right| ^{\frac p2} \\
&&= c_p h^{-2p}( B_1 + B_2).
\end{eqnarray*}
For the term $B_1$ we can write
\[
B_1 \le  E \left(  \sup_{x\in \mathbb{R}} \sup_{0\le  \sigma \le s\le r\le t} |A^h_{ \sigma,
r}(x) |^p  \int_0^s
 \left| \int_s^t \int_0^h \delta^h _{r,\sigma} (x) dxdr \right| ^p d\sigma  \right).
 \]
Applying Cauchy-Schwarz inequality and lemmas \ref{lema1} and \ref{lema2} in the Appendix we obtain
 \[
 B_1\le  C_{p,t,\epsilon} (t-s)^{\frac p2  }   h^{2p + \frac p2 -\epsilon}.
 \]
A similar estimate can be deduced for the term $B_2$. Finally, an application of the Garsia-Rodemich-Rumsey lemma allows us to conclude the proof.

\textit{Step 6 }  Let us prove (\ref{f2}).  We have, by It\^{o}'s formula  and in view of (\ref{d1}) and (\ref{d2})
\begin{eqnarray}
\langle M^{h}\rangle _{t} &=&48^2h^{-4}\int_{0}^{t}\Delta
_{r,h}^{2}dr=48^2h^{-4}\int_{0}^{t}\int_{0}^{r}(\Psi _{r,\sigma }^{h})^2d\sigma dr  \nonumber
\\
&&\quad +48^2\times 2h^{-4}\int_{0}^{t}\int_{0}^{r}\Psi ^h_{r,\sigma }\left(
\int_{\sigma }^{r}\Psi^h _{r,s}d\widehat{B}_{s}\right) d\widehat{B}_{\sigma   \nonumber
}dr \\
&:=& 48^2 h^{-4} \left( R^1_{t,h} + 2 R^2_{t,h} \right). \label{d3}
\end{eqnarray}%
We are going to see that only the first summand in the above expression with give a nonzero contribution to the limit. Consider first  the term $ R^1_{t,h}$. We can express $(\Psi_{r,\sigma }^{h})^2$ as
\[
(\Psi _{r,\sigma }^{h})^2=\int_{0}^{h}\int_{0}^{h}A^h_{\sigma ,r}(x)A^h_{\sigma
,r}(y)\delta _{r,\sigma }^{h}(x)\delta _{r,\sigma }^{h}(y)dxdy,
\]%
and, by It\^o's formula
\[
A^h_{\sigma ,r}(x)A^h_{\sigma ,r}(y)=\int_{\sigma }^{r}\delta
_{r,s}^{h}(x)\delta _{r,s}^{h}(y)ds+\int_{\sigma }^{r}A^h_{s,r}(x)\delta
_{r,s}^{h}(y)d\widehat{B}_{s}+\int_{\sigma }^{r}A^h_{s,r}(y)\delta
_{r,s}^{h}(x)d\widehat{B}_{s}.
\]%
Substituting the above equality in the expression of $(\Psi _{r,\sigma }^{h})^2
$ yields
\begin{eqnarray*}
&&R^1_{t,h}  = \int_{0}^{t}\int_{0}^{r}\int_{\sigma
}^{r}\int_{0}^{h}\int_{0}^{h}\delta _{r,s}^{h}(x)\delta _{r,s}^{h}(y)\delta
_{r,\sigma }^{h}(x)\delta _{r,\sigma }^{h}(y)dxdydsd\sigma dr \\
&&+ \int_{0}^{t}\int_{0}^{r}\int_{0}^{h}\int_{0}^{h}\left( \int_{\sigma
}^{r}A^h_{s,r}(x)\delta _{r,s}^{h}(y)d\widehat{B}_{s}\right) \delta _{r,\sigma
}^{h}(x)\delta _{r,\sigma }^{h}(y)dxdyd\sigma dr \\
&&+ \int_{0}^{t}\int_{0}^{r}\int_{0}^{h}\int_{0}^{h}\left( \int_{\sigma
}^{r}A^h_{s,r}(y)\delta _{r,s}^{h}(x)d\widehat{B}_{s}\right) \delta _{r,\sigma
}^{h}(x)\delta _{r,\sigma }^{h}(y))dxdyd\sigma dr \\
&=&\sum_{i=1}^{3}A _{t}^{i,h}.
\end{eqnarray*}%
Only the  first term in the above sum will give a nonzero contribution to the limit. Let us consider first this  term.
We have
\[
\int_{0}^{h}\delta _{r,s}^{h}(x)\delta _{r,\sigma }^{h}(x)dx\newline
=g_{h}(B_{r}-B_{s},B_{r}-B_{\sigma }),
\]%
where
\[
g_{h}(x,y)=(h-|x|-|y|)_{+}\mathbf{1}_{\{xy<0\}}+[(h-|x|)_{+}\wedge
(h-|y|)_{+}]\mathbf{1}_{\{xy\geq 0\}}.
\]%
As a consequence,
\begin{eqnarray*}
48^2 h^{-4}A_{t}^{1,h} &=&48^2h^{-4}\int_{0}^{t}\int_{0}^{r}\int_{\sigma }^{r}\
g_{h}(B_{r}-B_{s},B_{r}-B_{\sigma })^{2}\ dsd\sigma dr \\
&=&\frac 12 48^2h^{-4}\int_{0}^{t}\int_{0}^{r}\int_{0}^{r}\
g_{h}(B_{r}-B_{s},B_{r}-B_{\sigma })^{2}\ dsd\sigma dr \\
&=&\frac 12 48^2h^{-4}\int_{0}^{t}\left( \int_{\mathbb{R}^{2}}\
g_{h}(B_{r}-x,B_{r}-y)^{2}\ L_{r}^{x}L_{r}^{y}dxdy\right) dr.
\end{eqnarray*}%
As $h$ tends to zero this converges to $\frac 14 48^2%
\int_{0}^{t}(L_{r}^{B_{r}})^{2}dr=192\int_{\mathbb{R}}(L_{t}^{x})^{3}dx$.
This follows form the fact that%
\[
\int_{\mathbb{R}^{2}}g_{h}(x,y)^{2}dxdy=\frac{1}{2}h^{4}.
\]%
Let us show that the other terms $h^{-4}A_t^{2,h}$ and $ h^{-4}A ^{3,h}_t$ converge to zero  in $L^1(\Omega)$ as $h$ tends to zero. Using H\"older's inequality  with $\frac 1p + \frac 1q =1$  yields   
\begin{eqnarray*}
h^{-4}\left| A_{t}^{2,h}\right|  &\leq &h^{-4}\int_{0}^{t}\int_{0}^{r}\left(
\int_{0}^{h}\int_{0}^{h}\delta _{r,\sigma }^{h}(x)\delta _{r,\sigma
}^{h}(y)dxdy\right) ^{\frac{1}{p}} \\
&&\times \left( \int_{0}^{h}\int_{0}^{h}\left| \int_{\sigma
}^{r}A^h_{s,r}(x)\delta _{r,s}^{h}(y)d\widehat{B}_{s}\right| ^{q}dxdy\right) ^{%
\frac{1}{q}}d\sigma dr
 \\
&\leq &h^{-4}\int_{0}^{t}\int_{0}^{r}(h-|B_{r}-B_{\sigma }|)^{\frac 2p} _{+} %
 \left( \int_{0}^{h}\int_{0}^{h}\left| \int_{\sigma }^{r}A^h_{s,r}(x)\delta
_{r,s}^{h}(y)d\widehat{B}_{s}\right| ^{q}dxdy\right) ^{\frac{1}{q}}d\sigma dr   \\
&\leq &h^{-4}\sup_{r}\left( \int_{0}^{r}(h-|B_{r}-B_{\sigma }|)^{\frac 2p}_{+} d\sigma \right)  \\
&&\times \int_{0}^{t}\left( \sup_{\sigma }\int_{0}^{h}\int_{0}^{h}\left|
\int_{\sigma }^{r}A^h_{s,r}(x)\delta _{r,s}^{h}(y)d\widehat{B}_{s}\right|
^{q}dxdy\right) ^{\frac{1}{q}}dr \\
&=&h^{-4}B_{t}^{1,h}B_{t}^{2,h}.
\end{eqnarray*}%
The term $B_{t}^{1,h}$  can be estimated as follows
\[
\int_{0}^{r}(h-|B_{r}-B_{\sigma }|)^{\frac 2p}_{+} d\sigma =\int_{\mathbb{R%
}}(h-|B_{r}-x|)^{\frac 2p}_{+} L_{r}^{x}dx\leq \frac{2p}{p+2}\left(
\sup_{r\leq t,x\in \mathbb{R}}L_{r}^{x}\right) h^{\frac{2}{p}+1}.
\]%
Furthermore, for the term $B_{t}^{2,h}$ we can write
\begin{eqnarray*}
\left\| B_{t}^{2,h}\right\| _{q} &\leq &C  \sup_{r\le t}\left( \
\int_{0}^{h}\int_{0}^{h}E\left| \int_{0}^{r}A^h_{s,r}(x)\delta _{r,s}^{h}(y)d%
\widehat{B}_{s}\right| ^{q}dxdy\right) ^{\frac{1}{q}} \\
&\leq &C \sup_{r\le t} \left( \ \int_{0}^{h}\int_{0}^{h}E\left|
\int_{0}^{r}A^h_{s,r}(x)^{2}\delta _{r,s}^{h}(y)ds\right| ^{\frac{q}{2}%
}dxdy\right) ^{\frac{1}{q}} \\
&\leq &C  h^{\frac 1q}  \sup_{r\le t}\left(  E \left(  \sup_{x\in\mathbb{R}}\sup_{0\le s\leq r\le t}\left|
A^h_{s,r}(x)\right| ^{q}   \int_{0}^{r}\int_{0}^{h}\delta
_{r,s}^{h}(y)dsdy\right)   \right) ^{\frac{1}{q}} \\
&\leq &Ch^{\frac{3}{q}}\left(   E  \left( \sup_{x\in\mathbb{R}} \sup_{0\le s\leq
r\le t }\left| A^h_{s,r}(x)\right| ^{q}   \times  \sup_{ x\in \mathbb{R} }\sup _{0\le r \le t} %
 L_{r}^{x}\right)   \right) ^{\frac{1}{q}}.
\end{eqnarray*}%
Using Lemma \ref{lema1} in the Appendix with the exponent $q$  yields
\[
\left\| B_{t}^{2,h}\right\| _{q} \le Ch^{\frac 3q  +\frac 12-\frac \epsilon q}.
\]
 As a consequence,
 \[
h^{-4}E \left| A_{t}^{2,h}  \right|   \le C h^{-4} h^{\frac{2}{p}+1
  +\frac{2}{q}+\frac{1}{q}+\frac{1}{2}- \frac \epsilon q}=Ch^{\frac 1q -\frac 12 -\frac \epsilon q},
\]%
and for $0< q<2(1-\epsilon)$  this converges to zero.  In the same way we can show
that  $h^{-4}E \left| A_{t}^{3,h}  \right|  $ tends to zero as $h$
tends to zero.

It only remains to show that the   term $ h^{-4} R^2_{t,h}$  in the right-hand side of  (\ref{d3}) converges to zero.
Using Fubini's theorem we can write%
\[
R^2_{t,h}=\int_{0}^{t}\int_{0}^{r}\Psi ^h_{r,\sigma }\left( \int_{\sigma
}^{r}\Psi ^h_{r,s}d\widehat{B}_{s}\right) d\widehat{B}_{\sigma
}dr=\int_{0}^{t}\left( \int_{\sigma }^{t}\Psi ^h_{r,\sigma }\left(
\int_{\sigma }^{r}\Psi ^h_{r,s}d\widehat{B}_{s}\right) dr\right) d\widehat{B}%
_{\sigma },
\]%
hence,%
\begin{eqnarray*}
E(R^2_{t,h})^{2} &=&E\int_{0}^{t}\left( \int_{\sigma }^{t}\Psi^h _{r,\sigma
}\left( \int_{\sigma }^{r}\Psi ^h_{r,s}d\widehat{B}_{s}\right) dr\right)
^{2}d\sigma  \\
&=&E\int_{0}^{t}\int_{\sigma }^{t}\int_{\sigma }^{t}\Psi^h _{r,\sigma }\Psi
_{\rho ,\sigma }\left( \int_{\sigma }^{r}\Psi ^h_{r,s}d\widehat{B}_{s}\right)
\left( \int_{\sigma }^{\rho }\Psi ^h_{\rho ,s}d\widehat{B}_{s}\right) drd\rho
d\sigma  \\
&\leq &2E\int_{0}^{t}\int_{\sigma }^{t}\int_{\sigma }^{t}\left| \Psi
_{r,\sigma }\Psi^h _{\rho ,\sigma }\right| \left( \int_{\sigma }^{r}\Psi^h
_{r,s}d\widehat{B}_{s}\right) ^{2}drd\rho d\sigma  \\
&\leq &2\left\| \sup_{r}\left( \int_{0}^{r}\left| \Psi^h _{r,\sigma }\right|
d\sigma \right) \right\| _{a}\left\| \sup_{\sigma }\left( \int_{\sigma
}^{t}\left| \Psi^h _{\rho ,\sigma }\right| d\rho \right) \right\| _{b} \\
&&\times \left\| \left( \int_{0}^{t}\sup_{\sigma \leq r}\left(
\int_{\sigma }^{r}\Psi ^h_{r,s}d\widehat{B}_{s}\right) ^{2}dr\right)
\right\| _{c} 
\end{eqnarray*}%
with $\frac 1a +\frac 1b+ \frac 1c =1$.
Using lemmas \ref{lema1} and \ref{lema2} in the Appendix  we can show that the two first factors are  bounded by a constant times $h^{\frac{5}{2}  -\epsilon}$ for some   arbitrarily small $\epsilon>0$.   Using Doob's maximal inequality, the third
factor can be estimated by
\[
\sup_{0\le r\le t}   \left \| \int_0^r ( \Psi^h_{r,s} )^2 ds  \right\|_{\frac c2},
\]
and
\[
\int_0^r ( \Psi^h_{r,s} )^2 ds \le \sup _{x\in \mathbb{R}} \sup _{0\le s \le r \le t}
|A^h_{s, r} |^2 \int_0^r  (h- | B_r- B_s )_+^2   ds
\le  Ch^3 \sup _{x\in \mathbb{R}} \sup_ {0\le s \le r \le t}
|A^h_{s, r} |^2   \sup _{x\in \mathbb{R}} L_r^x.
\]
Finally, applying Lemma \ref{lema1},  we obtain
\[
h^{-8} E(R^2_{t,h})^{2} \le Ch ^{1- \delta}  
\]
for some   arbitrary small $\delta$. 
This completes the proof of Theorem \ref{thm1}.

\section{Appendix}
In this section we prove two technical results used in the paper.

\begin{lemma} \label{lema1}
Consider the random variable $A^h_{\sigma,r}$ introduced in  (\ref{l2}). Then,  for any $p\ge 2$ and $\epsilon \in (0, \frac p2)$ there exists a constant  $C_{t,p,\epsilon}  $ such that 
\[
E\sup_{x\in \mathbb{R}} \sup_{0\le \sigma \le r\le t} |A^h_{\sigma,r}(x)|^p
\le C_{t,p,\epsilon}   h^{\frac p2 -\epsilon}.
\]
\end{lemma}

\begin{proof}
By Tanaka's formula applied to the time reversed Brownian motion we can write
\begin{eqnarray*}
A_{ \sigma, r}(x)&=&-(B_r-B_\sigma +h) ^+   + (B_r-B_\sigma  ) ^+
+(B_r  -x+h) ^+   - (B_r-x  ) ^+  \\
&& +\frac 12 \left( L_\sigma^{B_r-x+h}  -L_\sigma^{B_r-x } -L_r^{B_r-x+h}+L_r^{B_r-x }\right) .
\end{eqnarray*}
Therefore,
\begin{equation}  \label{y2}
|A_{r,\sigma}(x)| \le 2h+\sup_{x\in \mathbb{R}}  \sup_{ 0\le r\le t} |L_r^{x+h} - L_r^{x }|.
\end{equation}
Finally, the result follows from the inequalities for the local time proved by Barlow and Yor in  \cite{BY}.
\end{proof}

\begin{lemma}  \label{lema2} Let $\delta^h_{r,\sigma} (x) $ be the random variable defined in  (\ref{l2}).  Then, for any $p\ge 2$ we there exists a constant $C_{t,p}$ such that for all $0\le s \le t$  
\[
E  \sup_{0\le \sigma \le s} \left |\int_s^t \int_0^h  \delta^h_{r,\sigma} (x) dxdr \right|^p \le h^{2p} C_{t,p} ( t-s) ^{\frac p2}.
\]
\end{lemma}
\begin{proof}
We can write 
\begin{eqnarray}
    \int_s^t \int_0^h \delta^h _{r,\sigma} (x) dxdr &= & \int_s^t  (h-|B_r-B_\sigma|)^+ dr  \nonumber \\
   && =  \int_\mathbb{R} (L_t^x-L_s^x) (h-|x-B_\sigma|)^+ dx   \nonumber \\
   &&\le   h^2 \sup_{x} (L_t^x-L_s^x).    
 \end{eqnarray}
 Finally,
 \[
 \sup_{x} (L_t^x-L_s^x) \le \sup_{x}  \int_s^t \delta_x (B_u -B_s) du,
 \]
 and  $\int_s^t \delta_x (B_u -B_s) du$ has the same distribution as $L_{t-s} ^x$, or, by the scaling properties of the local time, as $\sqrt{t-s} L_1 ^{x/\sqrt{t-s}}$, so
 \[
 E   \sup_{x} (L_t^x-L_s^x)^p \le ( t-s) ^{\frac p2} E\sup_{x} (L_1^x)^p.
 \]
 \end{proof}

\end{document}